\documentclass[12pt,leqno]{amsart}
\usepackage{euscript, amssymb, amsmath, amsthm}
\usepackage{epsfig}
\usepackage{graphicx}
\usepackage{subcaption}
\usepackage{yfonts}
\usepackage{caption}
\usepackage{longtable}
\usepackage{dcolumn}
\usepackage{setspace}
\usepackage[most]{tcolorbox}
\usepackage[colorlinks=true, citecolor=blue]{hyperref}
\definecolor{webred}{rgb}{0.75,0,0}
\definecolor{webgreen}{rgb}{0,0.75,0}
\definecolor{refkey}{gray}{0.75}
\numberwithin{equation}{section}

\newtheorem{theo}{Theorem}[section]
\newtheorem{lem}{Lemma}[section]

\newtheorem{Def}[theo]{Definition}
\theoremstyle{remark}
\newtheorem{rem}{Remark}[section]

\newcommand{\ep}{\varepsilon}
\def\R{{\mathbb{R}}}

\def\d{\displaystyle}
\def\e{{\varepsilon}}
\def\p{\partial}

\date{}


\tcbset{
    frame code={}
    center title,
    left=0pt,
    right=0pt,
    top=0pt,
    bottom=0pt,
    colback=gray!10,
    colframe=white,
    width=\dimexpr\textwidth\relax,
    enlarge left by=0mm,
    boxsep=5pt,
    arc=0pt,outer arc=0pt,
    }

\begin{document}

\title[Nonexistence result for the generalized Tricomi  equation]{blowup and lifespan estimate for the generalized
tricomi equation with the  scale-invariant damping and time derivative nonlinearity on exterior domain}
\author[M. Hamouda, M. A. Hamza  and B. Yousfi]{Makram Hamouda$^{1}$, Mohamed Ali Hamza$^{1}$ and Bouthaina Yousfi$^{2}$}
\address{$^{1}$  Department of Basic Sciences, Deanship of Preparatory Year and Supporting Studies, Imam Abdulrahman Bin Faisal University, P. O. Box 1982, Dammam, Saudi Arabia.}
\address{$^{2}$ Department of Mathemtics, Faclty of Sciences of Tunis, University Tunis El Manar, Tunisia.}

\medskip

\email{mmhamouda@iau.edu.sa (M. Hamouda)} 
\email{mahamza@iau.edu.sa (M.A. Hamza)}
\email{bouthaina.yousfi@fst.utm.tn (B. Yousfi)}

\pagestyle{plain}

%%%%%%%%%%%%%%%%%%%%%%%%%%%%
%%%%%%%%%%%%%%%%%%%%%%%%%%%%

\maketitle

\begin{abstract}

The article is devoted to investigating the initial boundary value problem for the damped wave equation in the \textit{scale-invariant} case with time-dependent speed of propagation on the exterior domain of a ball. By presenting suitable multipliers and applying the test-function technique, we study the blowup and the lifespan of the solutions to the problem with a derivative-type nonlinearity, namely
\begin{displaymath}
\d  u_{tt}-t^{2m}\Delta u+\frac{\mu}{t}u_t=|u_t|^p,
\quad \mbox{in}\ \Omega^{c}\times[1,\infty),
\end{displaymath}
that we associate with appropriate small initial data.
\end{abstract}

%%%%%%%%%%%%%%%%%%%%%%%%%%%%
%%%%%%%%%%%%%%%%%%%%%%%%%%%%

\section{Introduction}
\par\quad

We primarily consider the following initial boundary value problem for the damped wave equation on exterior domain, which is characterized by  a scale-invariant damping term and a time-dependent speed of propagation:
\begin{equation}
\label{G-sys}
\left\{
\begin{array}{l}
\d u_{tt}-t^{2m}\Delta u+\frac{\mu}{t}u_t=|u_t|^p,
\quad \mbox{in}\ \Omega^{c}\times[1,\infty),\\
u(x,1)=\e u_0(x),\ u_t(x,1)=\e u_1(x), \quad  x\in\Omega^{c},\\
u(x,t)\vert_{\p \Omega}=0, \quad t\geq 1,
\end{array}
\right.
\end{equation}
where $m$ is a nonnegative constant, $p>1$, $\mu \ge 0$ , $\ N \ge 1$ denotes the space dimension and the domain $\Omega^{c}=\R^{N} \backslash B_{1}(0)$ is the exterior of the unit ball. 

We set $B_{R}(0)=\left\{x \in \R^{N} : |x|<R \right\}$ with  $R>0$. The initial values $(u_{0}(x),u_{1}(x))$ belong to the sapce  $H^{1}(\Omega^{c})\times L^{2}(\Omega^{c})$.
 Moreover, the functions $u_{0}$ and $u_{1}$ are supposed to be positive and compactly supported on  $B_{R}(0) \backslash B_{1}(0)$ with  $R>1$. The constant 
 $\e>0$ is a sufficiently  small parameter describing the smallness of the initial data.\\

Let $\mu=m=0$ in \eqref{G-sys}. This   leads to  the semilinear wave equation with time-derivative nonlinearity on an exterior domain, which is  known to be characterized by a critical value of the exponent of the nonlinear term. The problem with derivative-type nonlinearity $|u_{t}|^{p}$, as in \eqref{G-sys} with $\mu=m=0$, we have the Glassey critical exponent $p_{G}$, which is given by
\begin{equation}\label{pg}
p_G=p_G(N):=1+\frac{2}{N-1},
\end{equation}
for which the blowup results are established for
 $p \leq p_{G}$; see  \cite{WH, Zhou2}. 

On the other hand and although it is not in the scope of the present work, we would like to mention the  case of power nonlinearity $|u|^{q}$ which is more or less well-understood. This case is marked by 
the Strauss critical exponent, denoted by $q_{S}$, which is the solution of the following quadratic equation:
\begin{equation}\label{qs}
    (N-1)q^{2}-(N+1)q-2=0.
\end{equation}
It is worth to mention that for $q \leq q_S$, there is no global solution and the lifespan estimates  have been  extensively studied. For more details, we refer the reader to the references \cite{WH, WH1, LLWW, LZ, LZ2, SW, Zhou2}. For $q> q_{S}$, the global existence is investigated in \cite{DMSZ, Hidano0, SSW, Wang}.

Now, we briefly review several previous blowup results about the Cauchy problem of the semilinear {\it damped} wave equation on exterior domains  with constant speed of propagation, namely $m = 0$. Y. Su et al. \cite{Su} acquire the upper bound lifespan estimates of solutions to the semilinear wave equation with damping term, and Neumann boundary conditions on an exterior domain in  dimension three. They show the blowup phenomenon by making use of the radial symmetry test function $\psi(r,t)=e^{-t}\frac{1}{r}e^{r}$, with $r=\sqrt{x_1^{2}+x_{2}^{2}+x_3^{2}}$. The aforementioned test function with two other radial functions, namely $\psi_{1}(r,t)=\rho_1(t)\frac{1}{r}e^{r}$ and $\psi_{2}(r,t)=\rho_{2}(t)\frac{1}{r}e^{r})$,  were also used in \cite{Fan} (with $\rho_i(t), i=1,2$ are functions to be determined). More precisely, X. Fan et al. \cite{Fan} apply the test function technique to study the same Cauchy problem but with a mass term. Employing the cut-off test function technique $(\psi(x,t)=\eta^{2p'}_{T}(t)\phi_{0}(x),\eta^{2p'}_{T}(t) \Phi(x,t) )$, C. Ren et al. \cite{Ren} investigate the blowup of solutions to the semilinear wave equation with a scale-invariant damping supplemented with a time-derivative nonlinearity  $|u_{t}|^{p}$  and combined nonlinearities  $|u_{t}|^{p} + |u|^{q}$, respectively.

Recently, several works have been devoted to the study of the Cauchy problem for semilinear damped wave equation with a time-dependent speed of propagation in the whole space, namely
\begin{equation}
\label{Sys1}
\left\{
\begin{array}{l}
\d u_{tt}-t^{2m}\Delta u+\frac{\mu}{t}u_t=|u_t|^p,
\quad \mbox{in}\ \R^{N}\times[1,\infty),\\
u(x,1)=\e u_0(x),\ u_t(x,1)=\e u_1(x), \quad  x\in\R^{N}.
\end{array}
\right.
\end{equation}
Concerning the blowup results and lifespan estimate of the solution to \eqref{Sys1} with $\mu=0$ and $m\geq 0$, it was
proven, in \cite{Our5}, that  the blowup holds for $p\leq p_{T}(N,m)$, where $p_{T}(N,m)$ is given by
\begin{equation}\label{pT}
    p_{T}(N,m):=1+\frac{2}{(m+1)(N-1)-m}, \quad \mbox{for}~ N \geq 2.
\end{equation}
Other blowup results for the combined nonlinearities are
obtained in \cite{CLP, HHP2}.
\par Motivated by the previous works in \cite{HHH, HHP1,HHP2}, our main objective is to study the lifespan estimates of solutions to \eqref{G-sys} which extends the domain of the problem \eqref{Sys1} to  exterior domains. In the same direction, we will investigate  in a subsquent work  \cite{HHY} the blowup results of solutions to problem \eqref{G-sys} with combined nonlinearities  $|u_{t}|^{p} + |u|^{q}$.

The main idea  in the present paper is the employment of adequate test functions and the construction of specific functionals which allow us to illustrate the blowup phenomenon and the upper bound lifespan estimate of the solution to the problem \eqref{G-sys}.

The remainder of this paper is organized as follows. In Section \ref{sec-main}, we introduce the definition of the weak solution of (\ref{G-sys}). Then, we state the main
theorem (Theorem \ref{blowup}) of our work. In Section \ref{aux}, we prove some technical lemmas to develop the proofs
of our results. Finally, Section \ref{sec-ut} is concerned with the proof of Theorem \ref{blowup}.

\section{Main Result}\label{sec-main}
\par

In this section, we will state our main result. But, before that we start by giving the definition of solution of (\ref{G-sys}) in the corresponding energy space.
\begin{Def}\label{def1}
Let $u_{0} \in H^{1}(\Omega^{c})$ and $u_{1} \in L^{2}(\Omega^{c})$. We say that $u$ is a weak solution of 
 (\ref{G-sys}) on $[1,T)$ if 
\begin{equation}\label{u-hyp}
 \left\{
\begin{array}{ll}
\d u\in \mathcal{C}([1,T),H^1(\Omega^{c}))\cap \mathcal{C}^1([1,T),L^2(\Omega^{c})) \ 
\\ \ u_t \in L^p_{loc}((1,T)\times \Omega^{c}),
\end{array}
\right.
\end{equation}
verifies, for all $\Phi\in \mathcal{C}_0^{\infty}(\Omega^{c}\times[1,T))$ and all $t\in[1,T)$, the following equality:
\begin{equation}
\label{energysol2}
\begin{array}{l}
\d\int_{\Omega^{c}}u_t(x,t)\Phi(x,t)dx-\int_{\Omega^{c}}u_t(x,1)\Phi(x,1)dx  -\int_1^t  \int_{\Omega^{c}}u_t(x,s)\Phi_t(x,s)dx \,ds\vspace{.2cm}\\
\d-\int_1^t  \int_{\Omega^{c}}s^{2m} \nabla u(x,s).\nabla \Phi(x,s) dx \,ds+\int_1^t  \int_{\Omega^{c}}\frac{\mu}{s}u_t(x,s) \Phi(x,s)dx \,ds\vspace{.2cm}\\
\d =\int_1^t \int_{\Omega^{c}}|u_t(x,s)|^p\Phi(x,s)dx \,ds,
\end{array}
\end{equation}
and the condition $u(x,1)=\varepsilon u_0(x)$ is fulfilled in $H^1(\Omega^{c})$. 
\end{Def}

The objective of our main result consists in deriving  the blowup region and the lifespan estimate of the solutions of \eqref{G-sys}. This is presented in the following
theorem.

\begin{theo}
\label{blowup}
Let $\mu \ge 0$, $m\ge 0$. Assume that  $(u_0,u_1) \in H^1(\Omega^{c}) \times L^2(\Omega^{c})$
and satisfy the following
conditions 
\begin{equation}\label{u0u1}
\left\{
\begin{aligned}
&u_{0}(x) \geq 0,~ u_{1}(x) \geq 0,~ a.e., \\
&u_0(x)=u_1(x)=0, ~for ~|x|>R, ~where ~R>1,
\\&u_0(x),~u_1(x) \not\equiv 0.
\end{aligned}
\right.
\end{equation}
%are non-negative and compactly supported functions   on  $B_{\R^N}(0,R)$ which do not vanish everywhere.
%\begin{equation}\label{CP}
    %\min (\mu-1,0)u_{0}(x)+u_{1}(x)>0.
%\end{equation}
Then,  there exists $\e_0=\e_0(u_0,u_1,N,R,p,m,\mu)>0$ such that for any $0<\e\le\e_0$ the solution $u$  to \eqref{G-sys} which satisfies 
$$\mbox{\rm supp}(u)\ \subset\{(x,t)\in\Omega^{c}\times[1,\infty): |x|\le R+\phi_m(t)\},$$
blows up in finite time $T_\e$, where
\begin{equation}\label{xi}
\phi_m(t):=\frac{t^{1+m}}{1+m}.
\end{equation} 
Furthermore, the upper bound of $T_\e$ is given,  for $N \ge 2$, by
\begin{displaymath}
T_\e \leq
\d \left\{
\begin{array}{ll}
 C \, \e^{-\frac{2(p-1)}{2-((1+m)(N-1)-m+\mu)(p-1)}}
 &
 \ \text{for} \
 1<p<p_{T}(N+\frac{\mu}{m+1},m), \vspace{.1cm}
 \\
 \exp\left(C\e^{-(p-1)}\right)
&
 \ \text{for} \ p=p_{T}(N+\frac{\mu}{m+1},m),
\end{array}
\right.
\end{displaymath}
 where $p_{T}(N,m)$ is defined by \eqref{pT}, and  for $N=1$, we have
\begin{displaymath}
T_\e \leq
\d \left\{
\begin{array}{ll}
 C \, \e^{-\frac{2(p-1)}{2-(\mu-m)(p-1)}}
 & \text{for} \
1< p < p_{\mu,m}:= \left\{
\begin{array}{ll}
1+\frac{2}{\mu-m} &\text{if} \ m<\mu \le m+\frac{2}{p-1}, \\
\infty  &\text{if} \ \mu \le m,
\end{array}
\right.
 \\
 \exp\left(C\e^{-(p-1)}\right)
&
 \ \text{for} \ 1<p=1+\frac{2}{\mu-m}.
\end{array}
\right.
\end{displaymath}
In the above lifespan estimates, the constant $C$ is positive and independent of $\e$.
\end{theo}

\begin{rem}
\textnormal{In view of the upper bound of the lifespan obtained in Theorem \ref{blowup}  for $N=1$, it is worth mentioning the competition between the damping and the tricomi terms. In fact, in the absence of  damping, the blowup of the solutions to \eqref{G-sys} in  dimension $N=1$ occurs for all $p>1$.}
\end{rem}

\begin{rem}
\textnormal{Note that the results in Theorem \ref{blowup} remain unchanged when adding a mass term in the Cauchy problem \eqref{G-sys}, this will be rigorously proven in the forthcoming work \cite{HHY}. Furthermore, replacing the damping term $\frac{\mu}{1+t} u_t$ in  \eqref{G-sys} by $b(t)u_t$, with $[b(t)-\mu (1+t)^{-1}]$ belonging to $L^1(0,\infty)$, leads to the same results as in Theorem \ref{blowup}. The proof is as well the same with the necessary adaptations.}
\end{rem}

\begin{rem}
\textnormal{The extension of the results in Theorem \ref{blowup} to a tricomi term with negative power, namely $-1<m<0$, can be done  by following the proofs in \cite{HHP1}, and this will be one of the subjects in the forthcoming work \cite{HHY}.}
\end{rem}

%%%%%%%%%%%%%%%%%%%%%%%%%%%%
%%%%%%%%%%%%%%%%%%%%%%%%%%%%
\section{Auxiliary results}\label{aux}
\par
In the present study, we will use an adequate test function as a principal tool. Thus, motivated by some previous works \cite{Fan,Ren,Zhou2}, we define the following positive test function 
\begin{equation}
\label{test11}
\psi(x,t):=\rho(t)\phi(x),
\end{equation}
where $\phi(x)$ will be introduced in Lemma \ref{lemphi1} below and $\rho(t)$ is a solution of
\begin{equation}\label{lambda}
\frac{d^2 \rho(t)}{dt^2}-t^{2m}\rho(t)-\frac{d}{dt}\left(\frac{\mu}{t}\rho(t)\right)=0, \quad t \ge 1.
\end{equation}
In fact, the expression of $\rho(t)$ is given by
\begin{equation}\label{rho}
\rho(t)=t^{\frac{\mu+1}{2}}K_{\frac{\mu-1}{2(1+m)}}(\phi_m(t)), \quad \forall \ t \ge 1,
\end{equation}
where $K_{\eta}$ is the modified Bessel function of second kind defined as
\begin{equation}\label{bessel}
K_{\eta}(t)=\int_0^\infty\exp(-t\cosh \zeta)\cosh(\eta \zeta)d\zeta,\ \eta\in \mathbb{R},
\end{equation}
and $\phi_m(t)$ is defined in \eqref{xi}; see Appendix A in \cite{HHH} for more details about the obtaining of the solution in  \eqref{rho}. Furthermore, we enumerate some useful properties of the function $\rho(t)$, in Appendix \ref{appendix1} below, that will be used in the proof of our main result.
\\Hence, we can easily see that the function $\psi(x,t)$ verifies the conjugate equation corresponding to the linear problem, namely we have 
\begin{equation}\label{lambda-eq}
\partial^2_t \psi(x, t)-t^{2m}\Delta \psi(x, t) -\frac{\partial}{\partial t}\left(\frac{\mu}{t}\psi(x, t)\right)=0.
\end{equation}
\color{black}

Throughout this work, we use a generic parameter $C$ to denote a positive constant that might be dependent on ($p,m,\mu,N,u_0,u_1,\ep_0$) but independent of $\ep$ and whose value might be different  from a line to another. However, when necessary, we will explicitly mention the expression of $C$ in terms of the aforementioned parameters.\\

In order to prove the principal result, we first present several related lemmas.

\begin{lem}[\cite{Zhou2}]
\label{lemphi1} There exists a function $\phi(x) \in C^{2}(\Omega^{c})$ satisfying
the following boundary value problem,
\begin{equation}\label{boundarypb1}
 \left\{
\begin{array}{ll}
\d \Delta\phi(x)=\phi(x),\quad \mbox{in}\ \Omega^{c} \subset \R^N,~ \mbox{for}\ N\ge1,\vspace{.2cm}\\
 \phi(x) \big{\vert}_{\p \Omega^{c}}=0,\\
\d  \phi(x) \to \int_{S^{N-1}}e^{x\cdot\omega}d\omega, \quad \text{as} ~|x| \to \infty.
\end{array}
\right.
\end{equation}
\end{lem}
Moreover, there exists a positive constant $C$ such that,
\begin{equation}\label{phi-equivalente}
    0\leq \phi(x) \leq C(1+|x|)^{-\frac{N-1}{2}} e^{|x|}, \quad \forall~ x \in \Omega^{c}, \quad \forall~ N \ge 1.
\end{equation}
\par
In the following lemma, we set a classical estimate result for the function $\psi(x,t)$ identified in \eqref{test11}.
\begin{lem}\label{lemmpsi1}
Let $R > 1$. Assume that $\phi$ satisfies the conditions in Lemma \ref{lemphi1} and $\psi(x,t)$ is as in \eqref{test11}. Then, for all $t \geq 1$, it holds that
\begin{equation}\label{psi1}
 \int_{\Omega^{c}\cap \{|x|\leq \phi_{m}(t)+R\}}\psi(x,t)dx \leq C\rho(t)e^{\phi_{m}(t)}(R+\phi_{m}(t))^{\frac{N-1}{2}},
\end{equation}
where $C=C(m,N,R)>0$.
\end{lem}
\begin{proof}
Let $x \in \Omega^{c}$. Using  (\ref{phi-equivalente}), we infer that
\begin{equation}\label{3.8'}
\begin{aligned}
\int_{\Omega^{c}\cap \{|x|\leq \phi_{m}(t)+R\}}\psi(x,t)dx &  = \int_{\Omega^{c}\cap \{|x|\leq \phi_{m}(t)+R\}}\rho(t)\phi(x)dx 
\\& \leq C\rho(t) \int_{ \{|x|\leq \phi_{m}(t)+R\}} |x|^{-\frac{(N-1)}{2}}e^{|x|}dx
\\&\leq C\rho(t) \int_{S^{N-1}} dw \int_{0}^{\phi_{m}(t)+R} r^{-\frac{N-1}{2}} e^{r} r^{N-1}dr
\\&\leq C\rho(t) e^{\phi_{m}(t)} \int_{0}^{\phi_{m}(t)+R} r^{\frac{N-1}{2}} e^{r-\phi_{m}(t)}dr.
\end{aligned}
\end{equation}
By performing an integration by parts for the last integral in the RHS of \eqref{3.8'}, it holds that
$$\begin{aligned}
\int_{\Omega^{c}\cap \{|x|\leq \phi_{m}(t)+R\}}\psi(x,t)dx &\leq C\rho(t) e^{\phi_{m}(t)}(\phi_{m}(t)+R)^{\frac{N-1}{2}}.
\end{aligned}$$
This proves Lemma \ref{lemmpsi1}.
\end{proof}
Recall that $\psi(x,t)$  is compactly supported and its expression is given by \eqref{test11}, and using  Lemma \ref{lemphi1}, we have 
\begin{equation}\label{psi_u}
\begin{aligned}
\int_{\Omega^{c}} \psi(x,s) \Delta u(x,s) dx &= \int_{\partial \Omega^{c}} \psi(x,s) [\nabla u(x,s) \cdot n] dS-\int_{\Omega^{c}} \nabla\psi(x,s)\cdot\nabla u(x,s) dx
\\&=-\int_{\partial \Omega^{c}} u(x,s)[\nabla\psi(x,s)\cdot n] dS+ \int_{\Omega^{c}} u(x,s) \Delta \psi(x,s) dx
\\&=\int_{\Omega^{c}} u(x,s) \Delta \psi(x,s) dx.
\end{aligned}
\end{equation}
Thus, substituting $\Phi(x, t)$ by $\psi(x, t)$ in \eqref{energysol2}, then using \eqref{psi_u}, we can see that \eqref{energysol2} is equivalent to 
\begin{equation}
\begin{array}{l}\label{energysol2-1}
\d \int_{\Omega^{c}}\big[u_t(x,t)\psi(x,t)- u(x,t)\psi_t(x,t)+\frac{\mu}{t}u(x,t) \psi(x,t)\big] dx \vspace{.2cm}\\
\d +\int_1^t  \int_{\Omega^{c}}u(x,s)\left[\psi_{tt}(x,s)-s^{2m}\Delta \psi(x,s) -\frac{\partial}{\partial s}\left(\frac{\mu}{s}\psi(x,s)\right)\right]dx \,ds\vspace{.2cm}\\
\d =\int_{1}^{t}\int_{\Omega^{c}}|u_t(x,s)|^p\psi(x,s)dx \, ds + \e \int_{\Omega^{c}}\big[-u_0(x)\psi_t(x,1)+\left(\mu u_0(x)+u_1(x)\right)\psi(x,1)\big]dx.
\end{array}
\end{equation}
\par
Now, we introduce the following functionals which are devoted to proving the blowup criteria. More precisely, let
\begin{equation}
\label{F1def}
\mathcal{U}_1(t):=\int_{\Omega^{c}}u(x, t)\psi(x, t)dx,
\end{equation}
and
\begin{equation}
\label{F2def}
\mathcal{U}_2(t):=\int_{\Omega^{c}}u_t(x,t)\psi(x, t)dx.
\end{equation}

In the next lemmas we will give the first lower bounds of 
 $\mathcal{U}_1(t)$ and $\mathcal{U}_2(t)$, respectively. More
precisely, we will prove that $\e^{-1}t^{m}\mathcal{U}_{1}(t)$ and $\e^{-1}\mathcal{U}_{2}(t)$ are coercive. 
\par We note here that the proof of Lemma \ref{F1} below is known in the literature; see e.g. \cite{HHH,HHP2}. For the purpose of making the presentation a self-contained reading, we give here the proof of this lemma in detail.

\begin{lem}
\label{F1}
Let $u$ be an energy solution of the system \eqref{G-sys} with initial data satisfying
the assumptions in Theorem \ref{blowup}. Then, there exists $T_0=T_0(m,\mu)>2$ such that 
\begin{equation}
\label{F1postive}
\mathcal{U}_1(t)\ge C_{\mathcal{U}_1}\, \e t^{-m}, 
\quad\text{for all}\ t \ge T_0,
\end{equation}
where $C_{\mathcal{U}_1}$ is a positive constant depending probably  on $u_0$, $u_1$, $N,m$, and $\mu$.
\end{lem}

\begin{proof} 
Let $t \in [1,T)$. 
Hence, using \eqref{lambda-eq} and \eqref{test11}, the identity \eqref{energysol2-1} becomes
\begin{equation}
\begin{array}{l}\label{eq5}
\d \int_{\Omega^{c}}\big[u_t(x,t)\psi(x,t)- u(x,t)\psi_t(x,t)+\frac{\mu}{t}u(x,t) \psi(x,t)\big]dx
\vspace{.2cm}\\
\d=\int_1^t\int_{\Omega^{c}}|u_t(x,s)|^p\psi(x,s)dx \, ds 
+\d \e \, C(u_0,u_1),
\end{array}
\end{equation}
where 
\begin{equation}\label{Cfg}
C(u_0,u_1):=\int_{\Omega^{c}}\big[\big(\mu\rho(1)-\rho'(1)\big)u_0(x)+\rho(1)u_1(x)\big]\phi(x)dx.
\end{equation}
Employing \eqref{rho} and \eqref{appendix7}, we get
\begin{equation}\label{Cfg1}
\mu\rho(1)-\rho'(1)=K_{\frac{\mu-1}{2(1+m)} +1}(\phi_m(1)),
\end{equation}
and consequently we deduce that 
\begin{align}\label{Cfg}
C(u_0,u_1)&=K_{\frac{\mu-1}{2(1+m)}}(\phi_m(1)) \int_{\Omega^{c}}  u_1(x)\phi(x)dx \\&\ +K_{\frac{\mu-1}{2(1+m)} +1}(\phi_m(1))\int_{\Omega^{c}}u_0(x)\phi(x)dx. \nonumber
\end{align}
Thus, the constant $C(u_0,u_1)$ is positive thanks to the fact that the function $K_{\eta}(t)$ is positive (see \eqref{Kmu} in Appendix \ref{appendix1}) and the sign of the initial data.\\ 
Using the definition of $\mathcal{U}_1$, given by \eqref{F1def},  and \eqref{test11},  the identity \eqref{eq5} yields
\begin{equation}
\begin{array}{l}\label{eq6}
\d \mathcal{U'}_1(t)+\Gamma(t)\mathcal{U}_1(t)=\int_1^t\int_{\Omega^{c}}|u_t(x,s)|^p\psi(x,s)dx \, ds +\e \, C(u_0,u_1),
\end{array}
\end{equation}
where 
\begin{equation}\label{gamma}
\Gamma(t):=\frac{\mu}{t}-2\frac{\rho'(t)}{\rho(t)}.
\end{equation}
Ignoring the nonlinear term in \eqref{eq6}, then multiplying the obtained equation by $\d \frac{t^\mu}{\rho^2(t)}$ and integrating on $(1,t)$, we infer that
\begin{align}\label{est-G1}
\mathcal{U}_1(t)
\ge\frac{\mathcal{U}_1(1)}{\rho^{2}(1)}\frac{\rho^2(t)}{t^\mu}+{\e}C(u_0,u_1)\frac{\rho^2(t)}{t^\mu}\int_1^t\frac{s^\mu}{\rho^2(s)}ds.
\end{align}
Having in mind the definition of  $\phi_m(t)$, given by \eqref{xi}, the positivity of $\mathcal{U}_1(1)$ and   \eqref{rho},  the estimate \eqref{est-G1} yields
\begin{align}\label{est-G1-1}
\mathcal{U}_1(t)
\ge {\e}C(u_0,u_1) t K^2_{\frac{\mu-1}{2(1+m)}}\left(\phi_m(t)\right)\int^t_{t/2}\frac{ds}{sK^2_{\frac{\mu-1}{2(1+m)}}\left(\phi_m(s)\right)}, \quad \forall \ t \ge 2.
\end{align}
Thanks to \eqref{Kmu}, we deduce that there exists $T_0=T_0(m,\mu)>2$ such that 
\begin{align}\label{est-double}
\phi_m(t)K^2_{\frac{\mu-1}{2(1+m)}}(\phi_m(t))>\frac{\pi}{4} e^{-2\phi_m(t)} \quad \text{and}  \quad \phi_m(t)^{-1}K^{-2}_{\frac{\mu-1}{2(1+m)}}(\phi_m(t))>\frac{1}{\pi} e^{2\phi_m(t)}, \ \forall \ t \ge T_0/2.
\end{align}
Combining \eqref{est-G1-1} and  \eqref{est-double}, and recalling \eqref{xi}, we get
\begin{align}\label{est-U-2}
 \mathcal{U}_1(t)
&\ge \e \frac{C(u_0,u_1)}{4}t^{-m}e^{-2\phi_m(t)}\int^t_{t/2}\phi_m'(s)e^{2\phi_m(s)}ds.
\end{align}
Hence, we conclude that
\begin{align}\label{est-G1-3}
\mathcal{U}_1(t)
\ge \e \kappa C(u_0,u_1)t^{-m}, \ \forall \ t \ge T_0,
\end{align}
where $\d \kappa=\kappa(m)$ is a positive constant.

The proof of Lemma \ref{F1} is thus completed.
\end{proof}

Let
\begin{equation}
\label{F1def1}
\mathcal{F}_1(t):=\int_{\Omega^{c}}u(x, t)\psi_0(x, t)dx, 
\end{equation}
and
\begin{equation}
\label{F2def1}
\mathcal{F}_2(t):=\int_{\Omega^{c}}u_t(x, t)\psi_0(x, t)dx,
\end{equation}
where, $\psi_0(x, t):=e^{-\frac{t^{m+1}-1}{m+1}}\phi(x)$ which satisfies 
\begin{equation}\label{psi0}
\left\{
\begin{array}{ll}
\d\partial_t \psi_{0}(x, t)=-t^m \psi_0(x, t),\vspace{.2cm}\\
\Delta\psi_{0}(x,t)=\psi_{0}(x,t),
\end{array}
\right.
\end{equation}
and $\phi(x)$ is given by \eqref{boundarypb1}. 

We present here the following lemma that will be used to prove that the functional
$\mathcal{U}_2(t)$ is positive for all $t\geq 1$.

\begin{lem}\label{F2+}
Under the same assumptions as in Theorem \ref{blowup}, it holds that
    \begin{equation}
    \mathcal{F}_{2}(t) \ge  \frac{{\e}}{2\mathcal{M}(t)}C_0(u_{0},u_{1}), \quad \forall \ t  \ge 1.
    \end{equation}
\end{lem}

\begin{proof}
First, we introduce the following multiplier
\begin{equation}
\label{test1}
\mathcal{M}(t):=t^{\mu}.
\end{equation}
Hence, considering $\mathcal{M}(t)\psi_{0}(x,t)$ as a test function in \eqref{energysol2}, and using \eqref{psi0}, we infer that
\begin{equation}
\begin{array}{l}\label{eq4}
\d \mathcal{M}(t)\int_{\Omega^{c}}u_t(x,t)\psi_0(x,t)dx
-\e\int_{\Omega^{c}}u_1(x)\psi_0(x,1)dx \vspace{.2cm}\\
\d+\int_1^t\mathcal{M}(s)s^m\int_{\Omega^{c}}\left\{
u_t(x,s) \psi_0(x,s)-s^{m}u(x,s)\psi_0(x,s)\right\}dx \, ds \vspace{.2cm}\\
\d=\int_1^t\mathcal{M}(s)\int_{\Omega^{c}}|u_t(x,s)|^p\psi_0(x,s)dx \, ds.
\end{array}
\end{equation}
Using the definition of $\mathcal{F}_{1}$, we observe that 
\begin{equation}\label{eq5-1-f1f2}
\d \mathcal{F}_1'(t)=\int_{\Omega^{c}}
u_t(x,t) \psi_0(x,t)dx -t^{m}\mathcal{F}_{1}(t),
\end{equation}
 then  the above equation with \eqref{eq4} yield
\begin{equation}
\begin{array}{l}\label{eq5-1-0}
\d \mathcal{M}(t)(\mathcal{F}_1'(t)+t^{m}\mathcal{F}_1(t))-\e\int_{\Omega^{c}}u_1(x)\psi_0(x,1)dx+\int_1^t\mathcal{M}(s)s^m \mathcal{F}_1'(s) ds
\vspace{.2cm}\\
\d =\int_1^t\mathcal{M}(s)\int_{\Omega^{c}}|u_t(x,s)|^p\psi_0(x,s)dx \, ds.
\end{array}
\end{equation}
It is straightforward to see that $\mathcal{F}_1(t)$ verifies 
\begin{equation}\label{eq5-1-new}
\d \int_1^t\mathcal{M}(s)s^m\mathcal{F}_1'(s) ds=- \int_1^t (\mathcal{M}(s)s^m)'\mathcal{F}_1(s) ds+\mathcal{M}(t)t^m\mathcal{F}_1(t)-\mathcal{F}_1(1).
 \end{equation}
Thus, by combining \eqref{eq5-1-0} and \eqref{eq5-1-new}, we get
\begin{equation}
\begin{array}{l}\label{eq5-1}
\d \mathcal{M}(t)(\mathcal{F}_1'(t)+2t^{m}\mathcal{F}_1(t))
-{\e}C_0(u_0,u_1) \vspace{.2cm}\\
\d=\int_1^t(\mathcal{M}(s)s^m)'\mathcal{F}_1(s) ds+\int_1^t\mathcal{M}(s)\int_{\Omega^{c}}|u_t(x,s)|^p\psi_0(x,s)dx \, ds,
\end{array}
\end{equation}
where 
$$C_0(u_0,u_1):=\int_{\Omega^{c}}\left\{u_0(x)+u_1(x)\right\}\phi(x)dx.$$
Hence, using the definitions of $\mathcal{F}_1$ and  $\mathcal{F}_2$, given  by \eqref{F1def1} and  \eqref{F2def1}, respectively, and the fact that \eqref{eq5-1-f1f2} can be rewritten as
 \begin{equation}\label{def231-bis}\d \mathcal{F}_2(t)= \mathcal{F}_1'(t) +t^{m}\mathcal{F}_1(t),\end{equation}
 the equation  \eqref{eq5-1} becomes
\begin{equation}
\begin{array}{l}\label{eq5bis1}
\d \mathcal{M}(t)(\mathcal{F}_2(t)+t^{m}\mathcal{F}_1(t))
-{\e}C_0(u_0,u_1)  \vspace{.2cm}\\
\d=\int_1^t(\mathcal{M}(s)s^m)'\mathcal{F}_1(s) ds+\int_1^t\mathcal{M}(s)\int_{\Omega^{c}}|u_t(x,s)|^p\psi_0(x,s)dx \, ds.
\end{array}
\end{equation}
Differentiating in time the identity \eqref{eq5bis1} and using \eqref{def231-bis}, we obtain
\begin{align}
\begin{array}{l}\label{F1+bis1}
\d \frac{d}{dt} \left\{\mathcal{F}_2(t)\mathcal{M}(t)\right\}+   2\mathcal{M}(t)t^{m}\mathcal{F}_2(t)\vspace{.2cm}\\
\d=\mathcal{M}(t)t^{m}\left(\mathcal{F}_2(t)+t^{m}\mathcal{F}_1(t)\right)+\mathcal{M}(t)\int_{\Omega^{c}}|u_t(x,t)|^p\psi_0(x,t)dx.
\end{array}
\end{align}
Thanks to \eqref{eq5bis1}, it is easy to see that
\begin{equation}\label{sigma1}
\d \mathcal{M}(t)(\mathcal{F}_2(t)+t^{m}\mathcal{F}_1(t)) \ge {\e}C_0(u_0,u_1)+\Sigma_1(t)
\end{equation}
where
\begin{equation}\label{sigma11}
\d \Sigma_1(t):=\d   t^{m}\int_1^t (\mathcal{M}(s)s^m)'\mathcal{F}_1(s) ds=\d (\mu+m)  t^{m}\int_1^t s^{\mu+m-1}\mathcal{F}_1(s) ds.
\end{equation}
By using the fact that $\mathcal{U}_1(t)=e^{\phi_m(t)-\phi_m(1)} \rho(t)\mathcal{F}_1(t)$ together with the positivity of $\mathcal{U}_1(t)$, in view of Remark \ref{rem-pos} below, we can easily obtain that   $\Sigma_1(t) \ge 0$.\\
Then, equation \eqref{sigma1} yields
\begin{equation}\label{tC0}
    \mathcal{M}(t)t^{m}(\mathcal{F}_2(t)+t^{m}\mathcal{F}_1(t)) \ge {\e}C_0(u_0,u_1)t^{m}.
\end{equation}
Now, plugging \eqref{tC0} in \eqref{F1+bis1} and employing the fact that $\mathcal{M}(t) \ge 1, \ \forall \ t \ge 1$, we obtain
\begin{equation}\label{F1+bis45}
\frac{d}{dt} \left\{\mathcal{F}_2(t)\mathcal{M}(t)\right\}+   2\mathcal{M}(t)t^{m}\mathcal{F}_2(t)
\ge {\e}C_0(u_0,u_1)t^{m},
\end{equation}
that we rewrite as follows:
\begin{align}\label{F1+bis4}
\frac{d}{dt} \left\{e^{2 \phi_m(t)}\mathcal{F}_2(t)\mathcal{M}(t)\right\} \ge {\e}C_0(u_0,u_1)t^{m}e^{2 \phi_m(t)}.
\end{align}
Therefore, by integrating in time the inequality  \eqref{F1+bis4}, we infer that
\begin{align}\label{F1+bis5}
e^{2 \phi_m(t)}\mathcal{M}(t) \mathcal{F}_2(t) \ge e^{\frac{2}{m+1}}\mathcal{F}_{2}(1)+{\e}C_0\int_{1}^{t}s^{m}e^{2 \phi_m(s)}ds .
\end{align}
Finally, we deduce that
\begin{align}\label{F1+bis6}
\mathcal{F}_2(t) \ge  \frac{1}{2\mathcal{M}(t)}{\e}C_0(u_{0},u_{1}).
\end{align}
This concludes the proof of Lemma \ref{F2+}.
\end{proof}
\begin{rem}\label{rem-pos}
Thanks to \eqref{est-G1} and the positivity of the initial data, the function $\mathcal{U}_1(t)$ is positive for all $t \ge 1$.
\end{rem}
We can now prove the following lemma.
\begin{lem}\label{F11}
Assume that the hypotheses in Theorem \ref{blowup} are fulfilled for $u$  a solution of the Cauchy problem  \eqref{G-sys}. Then, there exists $T_1=T_1(m,\mu)>T_{0}$ such that 
\begin{equation}
\label{F2postive}
\mathcal{U}_2(t)\ge C_{\mathcal{U}_2}\, \e, 
\quad\text{for all}\ t  \ge  T_1,
\end{equation}
where $C_{\mathcal{U}_2}$ is a positive constant depending probably  on $u_0$, $u_1$, $N,m$, and $\mu$.
\end{lem}
 
\begin{proof}
Let $t \in [1,T)$. Recall the definitions of $\mathcal{U}_1$  and $\mathcal{U}_2$, given respectively by \eqref{F1def} and \eqref{F2def}, the definition of $\psi$ as in \eqref{test11}, and the fact that 
 \begin{equation}\label{def23}\d \mathcal{U}_1'(t) -\frac{\rho'(t)}{\rho(t)}\mathcal{U}_1(t)= \mathcal{U}_2(t),\end{equation}
the identity  \eqref{eq6} becomes
\begin{equation}
\begin{array}{l}\label{eq5bis}
\d \mathcal{U}_2(t)+\left(\frac{\mu}{t}-\frac{\rho'(t)}{\rho(t)}\right)\mathcal{U}_1(t)
=\d \int_1^t\int_{\Omega^{c}}|u_t(x,s)|^p\psi(x,s)dx \, ds +\e \, C(u_0,u_1).
\end{array}
\end{equation}
Differentiating in time the equation \eqref{eq5bis} and using \eqref{lambda} and \eqref{def23}, we infer that
\begin{align}\label{F1+bis2}
\d \mathcal{U}_2'(t)+\left(\frac{\mu}{t}-\frac{\rho'(t)}{\rho(t)}\right)\mathcal{U}_2(t)-t^{2m}\mathcal{U}_1(t) =\int_{\Omega^{c}}|u_t(x,s)|^p\psi(x,t)dx. \nonumber
\end{align}
Taking into consideration the definition of $\Gamma(t)$, given by \eqref{gamma}, we conclude that 
\begin{equation}\label{G2+bis3}
\begin{array}{c}
\d \mathcal{U}_2'(t)+\frac{3\Gamma(t)}{4}\mathcal{U}_2(t)\ge\Sigma_2(t)+\Sigma_3(t)+\int_{\Omega^{c}}|u_t(x,t)|^p\psi(x,t)dx,
\end{array}
\end{equation}
where 
\begin{equation}\label{sigma1-exp}
\Sigma_2(t):=\d \left(-\frac{\rho'(t)}{2\rho(t)}-\frac{\mu}{4t}\right)\left(\mathcal{U}_2(t)+\left(\frac{\mu}{t}-\frac{\rho'(t)}{\rho(t)}\right)\mathcal{U}_1(t)\right),
\end{equation}
and
\begin{equation}\label{sigma2-exp}
\Sigma_3(t):=\d \left(t^{2m}+\left(\frac{\rho'(t)}{2\rho(t)}+\frac{\mu}{4t}\right) \left(\frac{\mu}{t}-\frac{\rho'(t)}{\rho(t)}\right) \right)  \mathcal{U}_1(t).
\end{equation}
Considering the asymptotic result \eqref{lambda'lambda1} and employing \eqref{eq5bis}, we deduce the existence of $\tilde{T}_1=\tilde{T}_1(m,\mu) \ge T_0$ such that
\begin{equation}\label{sigma11}
\d \Sigma_2(t) \ge C \, \e t^m + \frac{t^m}4 \int_1^t\int_{\Omega^{c}}|u_t(x,t)|^p\psi(x,s)dx \, ds, \quad \forall \ t \ge \tilde{T}_1. 
\end{equation}
Furthermore, from Lemma \ref{F1} and \eqref{lambda'lambda1}, we conclude the existence of $\tilde{T}_2=\tilde{T}_2(m,\mu) \ge \tilde{T}_1(m,\mu)$ ensuring that
\begin{equation}\label{sigma2}
\d \Sigma_3(t) \ge 0, \quad \forall \ t  \ge  \tilde{T}_2. 
\end{equation}
Now, by combining \eqref{G2+bis3}, \eqref{sigma11} and \eqref{sigma2}, we infer that
\begin{equation}\label{G2+bis4}
\begin{array}{l}
\d \mathcal{U}_2'(t)+\frac{3\Gamma(t)}{4}\mathcal{U}_2(t)\ge C_2 \, \e t^m+\int_{\Omega^{c}}|u_t(x,t)|^p\psi(x,t)dx \vspace{.2cm}\\
\d + \frac{t^m}4 \int_1^t\int_{\Omega^{c}}|u_t(x,t)|^p\psi(x,s)dx \, ds, \quad \forall \ t  \ge  \tilde{T}_2.
\end{array}
\end{equation}
By omiting the nonlinear term in \eqref{G2+bis4}, we obtain
\begin{equation}\label{G2+bis41}
\begin{array}{l}
\d \mathcal{U}_2'(t)+\frac{3\Gamma(t)}{4}\mathcal{U}_2(t)\ge C_2 \, \e t^m, \quad \forall \ t  \ge  \tilde{T}_2.
\end{array}
\end{equation}
Integrating \eqref{G2+bis41} over $(\tilde{T}_2,t)$ after multiplication  by $\frac{t^{3\mu/4}}{\rho^{3/2}(t)}$, we deduce that
\begin{align}\label{est-G111-bis}
\mathcal{U}_2(t)
\ge \mathcal{U}_2(\tilde{T}_2)\frac{\rho^{3/2}(t)}{\rho^{3/2}(\tilde{T}_2)t^{3\mu/4}}+C_2\,{\e} \frac{\rho^{3/2}(t)}{t^{3\mu/4}}\int_{\tilde{T}_2}^t\frac{s^{m+3\mu/4}}{\rho^{3/2}(s)}ds, \quad \forall \ t  \ge  \tilde{T}_2.
\end{align}
Now, observe that $\mathcal{U}_2(t)=e^{\phi_m(t)-\phi_m(1)} \rho(t) \mathcal{F}_2(t)$, where $\mathcal{F}_2(t)$ is given by \eqref{F2def1}.\\
Thus, using Lemma \ref{F2+} we conclude that $\mathcal{U}_{2}(t) \ge 0$ for all $t \ge 1$.\\
Therefore, using \eqref{rho} and the above observation, we infer that
\begin{equation}\label{G2sup}
\mathcal{U}_2(\tilde{T}_2)\frac{\rho^{3/2}(t)}{\rho^{3/2}(\tilde{T}_2)t^{3\mu/4}} \ge 0, \quad \forall \ t  \ge 1.
\end{equation}
Employing \eqref{G2sup}, the estimate \eqref{est-G111-bis} yields
\begin{equation}
   \mathcal{U}_{2}(t) \ge C\,{\e} \frac{\rho^{3/2}(t)}{t^{3\mu/4}}\int_{\tilde{T}_2}^t\frac{s^{m+3\mu/4}}{\rho^{3/2}(s)}ds, \quad \forall \ t  \ge  \tilde{T}_2. 
\end{equation}
Thanks to \eqref{est-double}, and the definitions of  $\rho(t)$ and $\phi_{m}(t)$ given, respectively, by \eqref{rho}
and \eqref{xi}, we infer the existence of  
%where $\phi_m(t)$ is defined by \eqref{xi}.\\
 $T_1=T_1(m,\mu):=2\tilde{T}_2$ such that the following estimate holds true
\begin{align}\label{est-G2-12}
\mathcal{U}_2(t)
&\ge \d  C\,{\e}  t^{-\frac{3}{4}m}e^{-\frac{3}{2}\phi_{m}(t)} \int^t_{t/2} s^{\frac{3}{4}m}\phi'_{m}(s) e^{\frac{3}{2}\phi_{m}(s)}ds.
\end{align}
A simple computation in the integral term in the above inequality gives that
\begin{align}\label{est-G2-123}
\mathcal{U}_2(t)
&\ge \d   C\,{\e} \left(1- e^{-\frac{3t^{m+1}}{2m+2}\left(1- (\frac12)^{m+1}\right)}\right).
\end{align}
Consequently,  we deduce that
\begin{align}\label{est-G1-2}
\mathcal{U}_2(t)
\ge  C_{\mathcal{U}_2}\,{\e}, \quad \forall \ t \ge T_1.
\end{align}

 Therefore, Lemma \ref{F11} is proven.
\end{proof}

%%%%%%%%%%%%%%%%%%%%%%%%%%%%
%%%%%%%%%%%%%%%%%%%%%%%%%%%%

%%%%%%%%%%%%%%%%%%%%%%%%5New Part
\section{Proof of Theorem \ref{blowup}.}\label{sec-ut}

This section is aimed to proving the main theorem in this article which exposes the blowup result and the lifespan estimate of the solution of \eqref{G-sys}. For this purpose, we will use the lemmas proven in Section \ref{aux}.

First, we define the following functional:
\[
\mathcal{H}(t):=
\int_{\tilde{T}_3}^t \int_{\Omega^{c}}|u_t(x,s)|^p\psi(x,s)dx ds
+C_3 \e,
\]
where $C_3=\min(C_2,8C_{\mathcal{U}_2})$; $C_{\mathcal{U}_2}$ is given by Lemma \ref{F11}, and and we
choose $\tilde{T}_3>T_1$ such that $$2t^{m}-\frac{3\Gamma(t)}{4}>0, \quad \forall \ t \ge\tilde{T}_3,$$
this can be maintained by \eqref{gamma} and \eqref{lambda'lambda1}.\\
Now, the next functional is
$$\mathcal{F}(t):=8 \,\mathcal{U}_2(t)-\mathcal{H}(t).$$
Then, using \eqref{G2+bis4}, we infer that
\begin{equation}\label{G2+bis6}
\begin{array}{rcl}
\d \mathcal{F}'(t)+\frac{3\Gamma(t)}{4}\mathcal{F}(t) &\ge& \d 7\int_{\Omega^{c}}|u_t(x,t)|^p\psi(x,t) dx+C_3 \e \left(8t^{m}-\frac{3\Gamma(t)}{4}\right) \vspace{.2cm}\\ &+&  \d
\left(2t^{m}-\frac{3\Gamma(t)}{4}\right)\int_{\tilde{T}_3}^t \int_{\Omega^{c}}|u_t(x,s)|^p\psi(x,s)dx ds\vspace{.2cm}\\
&\ge&0, \qquad \forall \ t \ge \tilde{T}_3.
\end{array}
\end{equation}
Multiplying \eqref{G2+bis6} by $\frac{t^{3 \mu/4}}{\rho^{3/2}(t)}$ and integrating over $(\tilde{T}_3,t)$, we deduce that
\begin{align}\label{est-G111}
 \mathcal{F}(t)
\ge \mathcal{F}(\tilde{T}_3)\frac{\tilde{T}_3^{3 \mu/4}\rho^{3/2}(t)}{t^{3 \mu/4}\rho^{3/2}(\tilde{T}_3)}, \quad \forall \ t \ge \tilde{T}_3,
\end{align}
where $\rho(t)$ is given by \eqref{rho}.\\
Therefore, we deduce that $\d \mathcal{F}(\tilde{T}_3)=8\mathcal{U}_2(\tilde{T}_3)-C_3 \e \ge 8\mathcal{U}_2(\tilde{T}_3)-8C_{\mathcal{U}_2}\e \ge 0$; the positivity of $\mathcal{F}(\tilde{T}_3)$ is  guaranteed in view of Lemma \ref{F11} and the definition of $C_3$, namely $C_3=\min(C_2,8C_{\mathcal{U}_2}) \le 8C_{\mathcal{U}_2}$. This implies that $\mathcal{F}(t)$ is positive for all $t \ge \tilde{T}_3$. Consequently, we obtain
\begin{equation}
\label{G2-est}
8\mathcal{U}_2(t)\geq \mathcal{H}(t), \quad \forall \ t \ge \tilde{T}_3.
\end{equation}
Thanks to  H\"{o}lder's inequality, \eqref{psi1} and \eqref{F2postive}, we can bound by below the nonlinear term as follows:
\begin{equation}
\begin{array}{rcl}
\d \int_{\Omega^{c}}|u_t(x,t)|^p\psi(x,t)dx &\geq&\d \mathcal{U}_2^p(t)\left(\int_{\{|x|\leq R+\phi_m(t)\}}\psi(x,t)dx\right)^{-(p-1)} \vspace{.2cm}\\ &\geq& C \mathcal{U}_2^p(t) \rho^{-(p-1)}(t)e^{-(p-1)\phi_m(t)}(\phi_m(t))^{-\frac{(N-1)(p-1)}2}.
\end{array}
\end{equation}
Using \eqref{est-rho}, the above estimate becomes
\begin{equation}\label{4.5}
\d \int_{\Omega^{c}}|u_t(x,t)|^p\psi(x,t)dx \geq C \, \mathcal{U}_2^p(t) t^{-\frac{\left[(N-1)(m+1)-m+\mu\right](p-1)}{2}}, \ \forall \ t \ge \tilde{T}_3.
\end{equation}
Then, combining  \eqref{G2-est} and \eqref{4.5}, we deduce that
\begin{equation}
\label{inequalityfornonlinearin}
\mathcal{H}'(t)\geq C \mathcal{H}^p(t) t^{-\frac{\left[(N-1)(m+1)-m+\mu\right](p-1)}{2}}, \quad \forall \ t \ge \tilde{T}_3,
\end{equation}
that we rewrite as below
\begin{equation}\label{4.7}
\frac{\mathcal{H}'(t)}{\mathcal{H}^{p}(t)} \geq C  t^{-\frac{\left[(N-1)(m+1)-m+\mu\right](p-1)}{2}}, \quad \forall \ t \ge \tilde{T}_3.
\end{equation}
Now, integrating the inequality \eqref{4.7} on $[t_1, t_2]$, for all $t_2 > t_1 \geq \tilde{T}_3$,  we obtain
\begin{equation}\label{4.8}
\mathcal{H}^{-(p-1)}(t_{2})-\mathcal{H}^{-(p-1)}(t_{1}) \geq C \int_{t_{1}}^{t_{2}}s^{-\frac{\left[(N-1)(m+1)-m+\mu\right](p-1)}{2}} ds,  \quad \forall  \ t_2 > t_1 \ge \tilde{T}_3.
\end{equation}
\\Using \eqref{G2-est} and \eqref{F2postive}, we can observe that $\mathcal{H}(\tilde{T}_3)=C_3 \e>0$, and 
\begin{equation}\label{Hp}
\mathcal{H}^{-(p-1)}(t) \leq C \e^{-(p-1)}, \quad \forall \ t \ge \tilde{T}_3.
\end{equation}
At this level, we distinguish two cases depending on the dimension $N$.
\\{\bf First case ($N=1$).}
\\For this value of $N$, and when $(\mu-m)(p-1)=2$, the estimate \eqref{4.8} implies that
\begin{equation}
    \mathcal{H}^{-(p-1)}(t_2) \geq \mathcal{H}^{-(p-1)}(t_1)+C ln(\frac{t_2}{t_1}),  \quad \forall  \ t_2 > t_1 \ge \tilde{T}_3.
\end{equation}
Then, employing \eqref{Hp},  we infer that
\begin{equation}\label{upperbound1}
    t_{2}\leq \exp(C \e^{-(p-1)}).
\end{equation}
Now, for the case $(\mu-m)(p-1)<2$, the estimate \eqref{4.8} yields
\begin{equation}
    \mathcal{H}^{-(p-1)} (t_{2}) \geq \mathcal{H}^{-(p-1)}(t_1)+C(t_{2}^{1-\frac{(\mu-m)(p-1)}{2}}-t_{1}^{1-\frac{(\mu-m)(p-1)}{2}}),  \quad \forall  \ t_2 > t_1 \ge \tilde{T}_3.
\end{equation}
Using \eqref{Hp} and choosing $\tilde{T}_4$ such that
\begin{equation}\label{T4}
    \tilde{T}_4=\max \big(\tilde{T}_3,C^{\frac{2}{2-(\mu-m)(p-1)}} \e^{\frac{-2(p-1)}{2-(\mu-m)(p-1)}} \big),
\end{equation}
we infer that,
\begin{equation}
    t_{2}\leq C \e^{\frac{-2(p-1)}{2-(\mu-m)(p-1)}}.
\end{equation}
{\bf Second case ($N\ge 2$).}
\\For $p=1+\frac{2}{\mu-m}$, we obtain the same upper bound as in  \eqref{upperbound1}. However, for $p<1+\frac{2}{\mu-m}$ we have 
\begin{equation}
 \mathcal{H}^{-(p-1)} (t_{2}) \geq C(t_{2}^{1-\frac{\left[(N-1)(m+1)-m+\mu\right](p-1)}{2}}-t_{1}^{1-\frac{\left[(N-1)(m+1)-m+\mu\right](p-1)}{2}}),  \quad \forall  \ t_2 > t_1 \ge \tilde{T}_3.    
\end{equation}
At this level, we set $\tilde{T}_5$ such that
\begin{equation}\label{T5}
    \tilde{T}_5=\max \big(\tilde{T}_3,(C\e^{-(p-1)})^\frac{2}{2-\left[(N-1)(m+1)-m+\mu\right](p-1)} \big).
\end{equation} 
Hence, using \eqref{Hp}, we deduce that
\begin{equation}
 t_{2}\leq C \e^{\frac{-2(p-1)}{2-\left[(N-1)(m+1)-m+\mu\right](p-1)} }  .
\end{equation}
This achieves the proof of Theorem \ref{blowup}. 

%%%%%%%%%%%%%%%%%
\appendix 
\section{}\label{appendix1}
In this appendix, we will list some properties of the function  $\rho(t)$ solution of \eqref{lambda}. Hence, following similar computations as in \cite[Appendix A]{HHH} (with $\nu$=0), we can write the expression of $\rho (t)$ as follows:
\begin{equation}\label{rho1}
\rho(t)=t^{\frac{\mu+1}{2}}K_{\frac{\mu-1}{2(1+m)}}(\phi_m(t)), \quad \forall \ t \ge 1,
\end{equation}
where
\begin{equation}\label{bessel1}
K_{\eta}(t)=\int_0^\infty\exp(-t\cosh \zeta)\cosh(\eta \zeta)d\zeta,\ \eta\in \mathbb{R},
\end{equation}
Thanks to \eqref{bessel1}, we infer that the function $\rho(t)$ is positive  on $[1,\infty)$. Then, from \cite{Gaunt},  the function $K_{\xi}(t)$ satisfies
\begin{equation}\label{Kmu}
K_{\eta}(t)=\sqrt{\frac{\pi}{2t}}e^{-t} (1+O(t^{-1})), \quad \text{as} \ t \to \infty.
\end{equation}
Combining \eqref{Kmu} and \eqref{rho1}, we obtain the existence of a constant $C_1$ such that   
\begin{equation}\label{est-rho}
C_1^{-1}t^{\frac{\mu-m}{2}} \exp(-\phi_m(t)) \le  \rho(t) \le C_1 t^{\frac{\mu-m}{2}} \exp(-\phi_m(t)), \quad \forall \ t \ge 1, 
\end{equation}
where $\phi_m(t)$ is defined by \eqref{xi}. Furthermore, for
all $m>0$, we have
\begin{equation}\label{lambda'lambda2}
   \frac{\rho'(t)}{\rho(t)}=\frac{\mu+1}{2t} +t^{m} \frac{K'_{\frac{\mu-1}{2(1+m)}}(\phi_{m}(t))}{K_{\frac{\mu-1}{2(1+m)}}(\phi_{m}(t))},
\end{equation}
Exploiting the well-known identity for the modified Bessel function of second kind, we have
\begin{equation}\label{appendix6}
\frac{d}{dy}K_{\nu}(y)=-K_{\nu+1}(y)+\frac{\nu}{y}K_{\nu}(y),
\end{equation}
and combining \eqref{lambda'lambda2} and \eqref{appendix6}, we obtain
\begin{equation}\label{appendix7}
   \frac{\rho'(t)}{\rho(t)}=\frac{\mu}{t}-t^{m}\frac{K_{1+\frac{\mu-1}{2(1+m)}}(\phi_{m}(t))}{K_{\frac{\mu-1}{2(1+m)}}(\phi_{m}(t))}.
\end{equation}
Thus, we deduce that
\begin{equation}\label{lambda'lambda1}
\d \lim_{t \to +\infty} \left(\frac{\rho'(t)}{t^m \rho(t)}\right)=-1.
\end{equation}
Finally, we refer the reader to \cite{HHH,Gaunt,Palmieri} for more details about the properties of the functions $\rho(t)$ and $K_{\eta}(t)$.
\color{black}

\bibliographystyle{plain}

\end{document}